\documentclass[11pt,reqno]{amsart}
\usepackage{amsfonts}
\usepackage{amsmath,amsthm,amssymb,amsfonts,amscd}
\usepackage{mathrsfs}
\usepackage{lipsum}
\usepackage{bbding}
\usepackage{graphicx,latexsym}
\usepackage{hyperref}
\usepackage{color}
\newcommand{\bea}{\begin{eqnarray}}
\newcommand{\eea}{\end{eqnarray}}
\newcommand{\bna}{\begin{eqnarray*}}
\newcommand{\ena}{\end{eqnarray*}}

\numberwithin{equation}{section}

\theoremstyle{plain}
\newtheorem{theorem}{Theorem}
\newtheorem{lemma}{Lemma}
\newtheorem{corollary}{Corollary}

\theoremstyle{definition}

\renewcommand{\Re}{\operatorname{Re}}

\begin{document}

\title{On an error term for the first moment of twisted $L$-functions }
\author{Xinyi He}

\address{School of Mathematics and Statistics, \ Shandong University, Weihai, \ Weihai, Shandong 264209, \ China}
\email{hxy1218@mail.sdu.edu.cn}
\date{\today}
\begin{abstract}
Let $f$ be a Hecke-Maass cusp form for the full modular group and 
let $\chi$ be a primitive Dirichlet character modulo a prime $q$. 
Let $s_0=\sigma_0+it_0$ with $\frac{1}{2}\leq\sigma_0<1$. 
We improve the error term for the first moment 
of $L(s_0,f\otimes\chi)\overline{L(s_0,\chi)}$ over 
the family of even primitive Dirichlet characters. 
As an application, we show that for any $t\in\mathbb{R}$, 
there exists a primitive Dirichlet character $\chi$ modulo $q$ 
for which $L(1/2+it,f\otimes\chi)L(1/2+it,\chi)\neq0$ if the 
prime $q$ satisfies $q\gg (1+|t|)^{\frac{543}{25}+\varepsilon}$.
\end{abstract}
\thanks{ This work is supported
by the National Natural Science Foundation of China (Grant No. 11871306).}
\keywords{The first moment, $L$-functions, error term.}
\maketitle
\section{Introduction}
 It is a problem of intensive study in analytic number theory to investigate the values of automorphic $L$-functions. Among others one is particularly interested in establishing asymptotic formula for moments of $L$-functions over families, since there are many applications such as nonvanishing problem of $L$-functions, the subconvexity problem and the problem of non-existence of Landau-Siegel zeros (see Duke \cite{duk}, Kowalski and Michel \cite{komi}, Conrey and Iwaniec \cite{jb}, and Iwaniec and Sarnak \cite{is}).

In \cite{dashan}, Das and Khan proved the following asymptotic formula
\bea\label{1.1r}
\sideset{}{^\dag}\sum_{\chi \bmod q\atop \chi(-1)=1}
L\left(\frac{1}{2},f \otimes \chi\right)
\overline{L\left(\frac{1}{2},\chi\right)}=
\frac{q-2}{2}L(1,f)
+O_{f,\varepsilon}\left(q^{\frac{7}{8}+\theta+\varepsilon}\right),
\eea
where $L(s,f\otimes\chi)$ is the $L$-function associated to a Hecke-Maass cusp form $f$ for the full modular group and a primitive Dirichlet character modulo a prime $q$, $L(s,\chi)$ is the Dirichlet $L$-function, and $L(1,f)$ denotes the $L$-function associated to $f$. Here and throughout the paper, the $\dagger$ means that the summation is over primitive characters and $\theta$ denotes the exponent towards the Ramanujan-Petersson conjecture for $f$,  which can be taken as $\theta=\frac{7}{64}$ due to Kim and Sarnak \cite{kim} (see also Liu \cite{liu} for an interesting related result). Recently, Sono \cite{sono} generalized Das and Khan's results to any complex number $s_0=\sigma_0+it_0$ with $1/2\leq\sigma_0< 1$. More precisely, he proved the following
asymptotic formula
\bea\label{1.1}
&&\sideset{}{^\dagger}\sum_{\chi (\bmod  q)\atop \chi(-1)=1}L(s_0,f\otimes\chi)\overline{L(s_0,\chi)}\notag\\
&=&\frac{q}{2}L(2\sigma_0,f)
+O_{f,\sigma_0,\varepsilon}\big(q^{1-\theta-2\sigma_0+\varepsilon}(q\tau)^{\theta+\varepsilon}
+
q^{-\frac{5}{4}+\frac{3}{2}\sigma_0+\varepsilon}(q\tau)^{\frac{3}{2}-\frac{3}{2}\sigma_0}\notag \\
&+&
(q\tau)^{1-\frac{1}{2}\sigma_0+\theta+\varepsilon}
+q^{-\frac{1}{2}}(q\tau)^{\frac{3}{2}-\frac{1}{2}\sigma_0+\varepsilon}
+q^{\frac{1}{2}}(q\tau)^{\frac{5}{4}-\frac{87}{52}\sigma_0+\frac{15}{26}\theta+\varepsilon}\notag \\
&+&q^{-1}(q\tau)^{\frac{5}{2}-\frac{5}{4}\sigma_0+\varepsilon}
+q^{-1}(q\tau)^{\frac{139}{52}-\frac{87}{52}\sigma_0+\frac{37}{26}\theta+\varepsilon}\big),
\eea
where here and throughout the paper, $\varepsilon>0$ is an arbitrarily small constant and $\tau:=|t_0|+3$. The main aim of our paper is to improve the error term in (\ref{1.1}). More precisely, we prove the following theorem.
\begin{theorem}
Let $f$ be a Hecke-Maass cusp form for $SL(2,\mathbb{Z})$. For prime values of  $q$ and $s_0=\sigma_0+it_0$ with $1/2\leq\sigma_0< 1$, we have
\begin{equation}\label{TH}
\sideset{}{^\dagger}\sum_{\chi \text(\bmod q) \atop \chi(-1)=1}
L(s_0,f\otimes\chi)\overline{L(s_0,\chi)}=\frac{q}{2}L(2\sigma_0,f)+O\left(\sum_{i=1}^{4}R_i\right),
\end{equation}
where the error terms $R_1$,...,$R_{4}$ are given by
\bna
&&R_1=q^{\frac{1}{4}+\varepsilon}\tau^{\frac{3}{2}(1-\sigma_0)},\ \
R_2=(q\tau)^{\frac{3}{2}(1-\sigma_0)+\theta+\varepsilon},\ \
R_3=q^{\frac{1}{2}}(q\tau)^{\frac{1}{2}(1-\sigma_0)+\varepsilon},\ \ \\
&&R_4=\tau(q\tau)^{\frac{3}{2}-\left(1+\frac{\sigma_0(1-2\theta)}{2(1+\theta)}\right)\sigma_0+\varepsilon}.
\ena
The implied constant may depend on $f$, but is independent of $q$ and $\tau$.
\end{theorem}
Note that the first error term in (\ref{1.1}) can be in fact dominated by the third one and for $\theta=\frac{7}{64}$, the fifth error term in (\ref{1.1}) is absorbed by the last error term in (\ref{1.1}). So we can ignore the first and fifth error terms in (\ref{1.1}). $R_1$ is the second error term in (\ref{1.1}). For the remaining terms, $R_2$ is superior to the sum of the third and the sixth error terms in (\ref{1.1}) for $\theta=\frac{7}{64}$. Obviously,  $R_3$ is superior to the fourth error term in (\ref{1.1}). Moreover, to compare $R_4$ and the last error term in (\ref{1.1}), we let
\bna
g(\sigma_0)=\frac{3}{2}-\left(1+\frac{\sigma_0(1-2\theta)}{2(1+\theta)}\right)\sigma_0.
\ena
We compute
\bna
&&g(\sigma_0)-\left(\frac{87}{52}-\frac{87}{52}\sigma_0+\frac{37}{26}\theta\right)\notag\\
&=&-\frac{1-2\theta}{2(1+\theta)}\sigma_0^2+\frac{35}{52}\sigma_0-\frac{9}{52}-\frac{37}{26}\theta\notag\\
&=&-\frac{25}{71}\sigma_0^2+\frac{35}{52}\sigma_0-\frac{547}{1664}\notag\\
&=&-\frac{25}{71}\left(\sigma_0-\frac{497}{520}\right)^2+\frac{25}{71}\left(\frac{497}{520}\right)^2-\frac{547}{1664}.
\ena
Therefore, for $1/2\leq\sigma_0< 1$, $g(\sigma_0)-\left(\frac{87}{52}-\frac{87}{52}\sigma_0+\frac{37}{26}\theta\right)\leq\frac{25}{71}\left(\frac{497}{520}\right)^2-\frac{547}{1664}<0$. Thus $R_4$ is superior to the last error term in (\ref{1.1}) for $\theta=\frac{7}{64}$. Assembling the above argument, we find that Theorem 1 improves (\ref{1.1}). In particular, for $\sigma_0=\frac{1}{2}$, we have
\bea\label{1.4}
\sideset{}{^\dagger}\sum_{\chi \text(\bmod q) \atop \chi(-1)=1}
L\left(\frac{1}{2}+it_0,f\otimes\chi\right)\overline{L\left(\frac{1}{2}+it_0,\chi\right)}=\frac{q}{2}L(1,f)
+O\left(q^{\frac{7}{8}+\frac{3\theta}{8(1+\theta)}+\varepsilon}\tau^{\frac{15}{8}+\frac{3\theta}{8(1+\theta)}+\varepsilon}\right).
\eea

To prove Theorem 1, we follow closely Sun \cite{sun}, where the case $q$ is a product of two primes is considered. In \cite{sun}, Sun also claimed that the error term in (\ref{1.1r}) can be improved to
$O(q^{\frac{7}{8}+\frac{3\theta}{8(1+\theta)}+\varepsilon})$. Notice that (\ref{1.4}) implies Sun's result.

Let
\bna
\mathbf{M}(\sigma_0)=\max \left\{2(1-\sigma_0),\ \frac{3-3\sigma_0+2\theta}{3\sigma_0-1-2\theta}, \ \frac{1-\sigma_0}{\sigma_0}, \ \frac{5+5\theta-\sigma_0(2+2\theta+\sigma_0-2\sigma_0\theta)}{-1-\theta+\sigma_0(2+2\theta+\sigma_0-2\sigma_0\theta)}\right\}.
\ena
Here we call that $\theta=7/64$. It can be easily confirmed that the main term in (\ref{TH}) dominates the error terms if $q\gg_f \tau^\mathbf{{M}(\sigma_0)+\varepsilon}$. Moreover, we have
\bna
\mathbf{M}\left(\frac{1}{2}\right)=\max\left\{1, \ \frac{3+4\theta}{1-4\theta}, \ \frac{15+18\theta}{1-2\theta}\right\}=\frac{543}{25}.
\ena
Hence we have the following result.
\begin{corollary}
For a Hecke-Maass cusp form $f$ of $SL(2,\mathbb{Z})$, there exists a primitive Dirichlet character $\chi$ modulo $q$ for which $L(s_0, f\otimes\chi)$ and $L(s_0, \chi)$ do not vanish if the prime $q$ satisfies $q\gg_f \tau^{\mathbf{M}(\sigma_0)+\varepsilon}$. In particular, for any $t\in\mathbb{R}$, there exists a primitive Dirichlet character $\chi$ modulo $q$ for which $L(1/2+it, f\otimes\chi)$ and $L(1/2+it, \chi)$ do not vanish if the prime $q$ satisfies $q\gg_f (1+|t|)^{543/25+\varepsilon}$.
\end{corollary}
Corollary 1 improves the result of Sono \cite{sono} who proved that $L(1/2+it,f\otimes\chi)L(1/2+it,\chi)\neq0$ for some primitive Dirichlet character modulo $q$ for prime values $q$ satisfying $q\gg_f (1+|t|)^{255+\varepsilon}$.
\section{Preliminaries}
Let $\chi$ be an even primitive Dirichlet character modulo $q$. For $\Re(s)>1$, we define the Dirichlet $L$-function
\bna
L(s,\chi)=\sum_{n\geq 1}\chi(n)n^{-s},
\ena
which has analytic continuation to all $s\in \mathbb{C}$ and satisfies the functional equation
\bna
\Lambda(s,\chi)=\frac{\tau(\chi)}{\sqrt{q}}\Lambda(1-s,\overline{\chi}),
\ena
where
\bna
\Lambda(s,\chi)=\left(\frac{q}{\pi}\right)^{\frac{s}{2}}\Gamma\left(\frac{s}{2}\right)L(s,\chi),
\ena
and $\tau(\chi)=\sideset{}{^\ast}\sum\limits_{a (\text{{\rm mod }} q)}\chi(a) e(a/q)$ is the Gauss sum. Hereafter the notation $\sideset{}{^\ast}\sum\limits_{a (\text{{\rm mod }} q)}$ means $\sum\limits_{a (\text{{\rm mod }} q)\atop (a,q)=1 }$.

Let $f$ be an even Hecke-Maass cusp form for $SL(2,\mathbb{Z})$ with Laplace eigenvalue $\frac{1}{4}+T_f^2$, $T_f\in \mathbb{R}$. Let $\lambda_f(n)$ be the $n$-th Fourier coefficient of $f$. For $\Re(s)>1$, we define the Dirichlet twist of Hecke-Maass $L$-function
\bna
 L(s,f\otimes\chi)=\sum_{n\geq 1}\lambda_f(n)\chi(n)n^{-s},
\ena
which has analytic continuation to the whole complex plane and satisfies the function equation
\bna
 \Lambda(s,f\otimes\chi)=\frac{\tau(\chi)^2}{q}\Lambda(1-s,f\otimes\overline{\chi}),
\ena
where
\bna
\Lambda(s,f\otimes\chi)=\left
(\frac{q}{\pi}\right)^{s}\Gamma\left(\frac{s+iT_f}{2}\right)\Gamma\left(\frac{s-iT_f}{2}\right)L(s,f\otimes\chi).
\ena

To prove Theorem 1, we need the approximate functional equations for $L(s,\chi)$ and $L(s,f\otimes\varphi)$. We quote the follow results of Sono \cite{sono} (see Lemmas 2.6 and 2.7 in \cite{sono}).
\begin{lemma}
Let $G(u)=e^{u^{2}}$. For $s_0=\sigma_0+it_0$ with $1/2\leq \sigma_0<1$ and for $\chi$ an even primitive Dirichlet character of modulus $q$, we obtain
\bea\label{V}
L(s_0,\chi)=\sum_{n\geq 1}\frac{\chi(n)}{n^{s_0}}V_{s_0}\left(\frac{n}{\sqrt{q}}\right)+\tau(\chi)q^{-s_0}\frac{\gamma(1-s_0)}{\gamma(s_0)}
\sum_{n\geq 1}\frac{\overline{\chi}(n)}{n^{1-s_0}}V_{1-s_0}\left(\frac{n}{\sqrt{q}}\right),
\eea
where $\gamma(s):=\pi^{-s/2}\Gamma(s/2)$ and
\bna
V_{s_0}(x)=\frac{1}{2\pi i}\int_{(2)}(\sqrt{\pi}x)^{-u}\frac{\Gamma(\frac{s_0+u}{2})}{\Gamma(\frac{s_0}{2})}G(u)\frac{\mathrm{d}u}{u}.
\ena
The function $V_{s_0}(x)$ satisfies
\bna
&&V_{s_0}(x)\ll_A \min\{1,(x/\sqrt{\tau})^{-A}\}, \\
&&V^{(l)}_{s_0}(x)\ll x^{-l}.
\ena
The implied constants depend only on $\sigma_0$, $A$, and $l$.
\end{lemma}
\begin{lemma}
Let $f$ be an even Hecke-Maass form for $SL(2,\mathbb{Z})$. For $s_0=\sigma_0+it_0$ with $1/2\leq \sigma_0<1$ and for $\chi$ an even primitive Dirichlet character of modulus $q$, we obtain
\bea\label{W}
L(s_0,f\otimes\chi)&=&\sum_{n\geq 1}\frac{\lambda_f(n)\chi(n)}{n^{s_0}}W_{s_0}\left(\frac{n}{q}\right)\notag\\
&+&
\tau(\chi)^{2}q^{-2s_0}\frac{\widetilde{\gamma}(1-s_0)}{\widetilde{\gamma}(s_0)}
\sum_{n\geq 1}\frac{\lambda_f(n)\overline{\chi}(n)}{n^{1-s_0}}W_{1-s_0}\left(\frac{n}{q}\right),
\eea
where ${\widetilde{\gamma}(s)}:=\pi^{-s}\Gamma\left(\frac{s+iT_f}{2}\right)\Gamma\left(\frac{s-iT_f}{2}\right)$ and
\bna
W_{s_0}(x)=\frac{1}{2\pi i}\int_{(2)}
(\pi x)^{-u}\frac{\Gamma(\frac{s_0+u+iT_f}{2})\Gamma(\frac{s_0+u-iT_f}{2})}{\Gamma(\frac{s_0+iT_f}{2})\Gamma(\frac{s_0-iT_f}{2})}G(u)\frac{\mathrm{d}u}{u}.
\ena
The function $W_{s_0}(x)$ satisfies
\bna
&&W_{s_0}(x)\ll_A \min\{1,(x/\tau)^{-A}\}, \\
&&W^{(l)}_{s_0}(x)\ll x^{-l}.
\ena
The implied constants depend only on $\sigma_0$, $A$, and $l$.
\end{lemma}
We need the Rankin-Selberg estimate (see Proposition 19.6 in Duke, Friedlander and Iwaniec \cite{duk2}).
\begin{lemma}\label{1}
For any $\varepsilon> 0$, we have
\bna
\sum_{n\leq x}|\lambda_f(n)|^2\ll_{f,\varepsilon} x^{1+\varepsilon}.
\ena
\end{lemma}
Moreover, we also need the Wilton-type bound (see Iwaniec \cite{ikk}, Theorem 8.1).
\begin{lemma}\label{e}
For any $\alpha\in \mathbb{R}$, we have
\bna\label{e_1}
\sum_{n\leq N}\lambda_f(n) e(\alpha n)\ll_{f,\varepsilon} N^{1/2+\varepsilon},
\ena
uniformly in $\alpha$.
\end{lemma}

The following Voronoi formula can be found in Miller and Schmid \cite{Misch} (see also Godber \cite{god}, Theorem 3.2).
\begin{lemma}\label{Vor}
Let $\psi$ be a fixed smooth function with compact support on $\mathbb{R^{+}}$, $c$ be a positive integer and $d$, $\overline{d}\in \mathbb{Z} $ with $(c,d)=1$ and $d\overline{d}\equiv 1(\bmod c)$. Then we have
\bna
\sum_{n\geq 1}\lambda_f(n) e\left(\frac{n\overline{d}}{c}\right)\psi\left(\frac{n}{N}\right)
=c\sum_\pm\sum_{n\geq 1}\frac{\lambda_f(n)}{n}e\left(\pm\frac{nd}{c}\right)\Psi^{\pm}\left(\frac{n N}{c^{2}}\right),
\ena
where for $\sigma > -1$,
\bna
\Psi^{\pm}(x)=\frac{1}{2\pi i}\int_{(\sigma)}(\pi^{2}x)^{-s}G^{\pm}(s){\widetilde{\psi}}(s)\mathrm{d}s.
\ena

Here ${\widetilde{\psi}}(s)=\int_{0}^{\infty}\psi(x)x^{s-1}dx$ is the Mellin transform of $\psi(x)$ and
\bea\label{G}
2\pi G^{\pm}(s)=\frac{\Gamma(\frac{1+s+iT_f}{2})\Gamma(\frac{1+s-iT_f}{2})}
{\Gamma(\frac{-s+iT_f}{2})\Gamma(\frac{-s-iT_f}{2})}\pm
\frac{\Gamma(\frac{1+s+iT_f+1}{2})\Gamma(\frac{1+s+iT_f+1}{2})}
{\Gamma(\frac{-s+iT_f+1}{2})\Gamma(\frac{-s-iT_f+1}{2})}.
\eea
\end{lemma}
\begin{lemma}\label{q}
If $\chi \ (\bmod \ q) $ is an even primitive character, then
\bea\label{q_1}
\tau(\overline{\chi})\tau(\chi)=q.
\eea
\begin{proof}
By (3.15) of Iwaniec and Kowalski \cite{ik}, $\tau(\chi)\tau(\overline{\chi})=\chi(-1)q$. Then (\ref{q_1})
follows since $\chi$ is even.
\end{proof}
\end{lemma}
\section{ Proof of Theorem 1}
\setcounter{equation}{0}
\medskip
 Applying the approximate functional equations in (\ref{V}) and (\ref{W}), we have
\bea\label{sum}
&&\sideset{}{^\dagger}\sum_{\chi (\bmod q) \atop \chi(-1)=1}
L(s_0,f\otimes\chi)\overline{L(s_0,\chi)}\notag\\
&=&\sideset{}{^\dagger}\sum_{\chi (\text{{\rm mod }} q) \atop \chi(-1)=1}
\left[\sum_{n=1}^{\infty}\frac{\lambda_f(n)\chi(n)}{n^{s_0}}W_{s_0}\left(\frac{n}{q}\right)
+\tau(\chi)^{2}q^{-2s_0}\frac{\widetilde{\gamma}(1-s_0)}{\widetilde{\gamma}(s_0)}
\sum_{n=1}^{\infty}\frac{\lambda_f(n)\overline{\chi}(n)}{n^{1-s_0}}W_{1-s_0}\left(\frac{n}{q}\right)\right]\notag\\
&\times&\left[\sum_{m=1}^{\infty}\frac{\overline{\chi}(m)}{m^{\overline{s_0}}}V_{\overline{s_0}}\left(\frac{m}{\sqrt{q}}\right)+
\tau(\overline{\chi})q^{-\overline{s_0}}\frac{\gamma(1-\overline{s_0})}{\gamma(\overline{s_0})}
\sum_{m=1}^{\infty}\frac{\chi(m)}{m^{1-\overline{s_0}}}V_{1-\overline{s_0}}\left(\frac{m}{\sqrt{q}}\right)\right]\notag\\
&:=&S_1+S_2+S_3+S_4,
\eea
where
\bea S_1&=&\sum_{m=1}^{\infty}\frac{1}{m^{\overline{s_0}}}V_{\overline{s_0}}\left(\frac{m}{\sqrt{q}}\right)
\sum_{n=1}^{\infty}\frac{\lambda_f(n)}{n^{s_0}}W_{s_0}\left(\frac{n}{q}\right)
\sideset{}{^\dagger}\sum_{\chi(\bmod q) \atop \chi(-1)=1}\overline{\chi}(m)\chi(n),\notag\\
S_2&=&q^{-\overline{s_0}}\frac{\gamma(1-\overline{s_0})}{\gamma(\overline{s_0})}\sum_{m=1}^{\infty}\frac{1}{m^{1-\overline{s_0}}}V_{1-\overline{s_0}}\left(\frac{m}{\sqrt{q}}\right)
\sum_{n=1}^{\infty}\frac{\lambda_f(n)}{n^{s_0}}W_{s_0}\left(\frac{n}{q}\right)
\sideset{}{^\dagger}\sum_{\chi (\text{{\rm mod }} q) \atop \chi(-1)=1}\chi(m)\chi(n)\tau(\overline{\chi}),\notag\\
S_3&=&q^{-2s_0-\overline{s_0}}\frac{\gamma(1-\overline{s_0})\widetilde{\gamma}(1-s_0)}{\gamma(\overline{s_0})\widetilde{\gamma}(s_0)}
\sum_{m=1}^{\infty}\frac{1}{m^{1-\overline{s_0}}}V_{1-\overline{s_0}}\left(\frac{m}{\sqrt{q}}\right)
\sum_{n=1}^{\infty}\frac{\lambda_f(n)}{n^{1-s_0}}W_{1-s_0}\left(\frac{n}{q}\right)\notag \\
&\times&\sideset{}{^\dagger}\sum_{\chi (\bmod q) \atop \chi(-1)=1}\chi(m)\overline{\chi}(n)\tau(\overline{\chi})\tau(\chi)^{2},\notag\\
S_4&=&q^{-2s_0}\frac{\widetilde{\gamma}(1-s_0)}{\widetilde{\gamma}(s_0)}
\sum_{m=1}^{\infty}\frac{1}{m^{\overline{s_0}}}V_{\overline{s_0}}\left(\frac{m}{\sqrt{q}}\right)
\sum_{n=1}^{\infty}\frac{\lambda_f(n)}{n^{1-s_0}}W_{1-s_0}\left(\frac{n}{q}\right)
\sideset{}{^\dagger}\sum_{\chi (\bmod q) \atop \chi(-1)=1}\overline{\chi}(m)\overline{\chi}(n)\tau(\chi)^{2}.\notag
\eea
In the subsequent sections, we will evaluate $S_i, \ i=1,2,3,4$.
\section{Estimation of $S_1$}
\begin{lemma}\label{jmh}
For any $\varepsilon >0$, we have
\bna\label{lemma1}
S_1=\frac{q}{2}L(2\sigma_0,f)+O\left(
q^{\frac{1}{4}+\varepsilon}\tau^{\frac{3}{2}(1-\sigma_0)}+
(q\tau)^{\frac{3}{2}(1-\sigma_0)+\theta+\varepsilon}\right).
\ena
\end{lemma}
\begin{proof}
We write
\bna
S_1=\sum_{m=1}^{\infty}\frac{1}{m^{\overline{s_0}}}V_{\overline{s_0}}\left(\frac{m}{\sqrt{q}}\right)
\sum_{n=1}^{\infty}\frac{\lambda_f(n)}{n^{s_0}}W_{s_0}\left(\frac{n}{q}\right)
\sideset{}{^\dagger}\sum_{\chi(\bmod q) \atop \chi(-1)=1}\overline{\chi}(m)\chi(n).
\ena
By the orthogonality property of Dirichlet characters, we have for $(mn,q)=1$,
\bna
\sideset{}{^\dagger}\sum_{\chi
(\bmod  q) \atop \chi(-1)=1} \chi(n)\overline{\chi}(m)
&=&
\frac{1}{2}\sideset{}{^\dagger}\sum_{\chi (\bmod  q)}\left[1+\chi(-1)\right]\chi(n)\overline{\chi}(m)\notag\\
&=&
\frac{1}{2}\sum_{\pm}\sum_{\chi (\bmod  q)}\chi(\pm n)\overline{\chi}(m)\notag\\
&=&
\frac{1}{2}\sum_{\pm}\left[\phi(q)\mathbf{1}_{m\equiv\pm n (\bmod q)}-1\right].
\ena
Therefore,
\bea\label{S_1}
S_1
&=&\frac{\phi(q)}{2}\sum_{\pm}\sum_{m\geq 1\atop (m,q)=1}\frac{1}{m^{\overline{s_0}}}V_{\overline{s_0}}\left(\frac{m}{\sqrt{q}}\right)
\mathop{\sum_{n\geq 1\atop n\equiv\pm m\text(\bmod q) }} \frac{\lambda_f(n)}{n^{s_0}}W_{s_0}\left(\frac{n}{q}\right) \notag \\
&-&
\sum_{m\geq 1\atop (m,q)=1}\frac{1}{m^{\overline{s_0}}}V_{\overline{s_0}}\left(\frac{m}{\sqrt{q}}\right)
\mathop{\sum_{n\geq 1\atop (n,q)=1}}\frac{\lambda_f(n)}{n^{s_0}}W_{s_0}\left(\frac{n}{q}\right) \notag \\
&=&\frac{\phi(q)}{2}S_{11}-S_{12},
\eea
say. Trivially, we have
\bea\label{S_12}
S_{12}
&=&
\sum_{m\geq 1\atop (m,q)=1}\frac{1}{m^{\overline{s_0}}}V_{\overline{s_0}}\left(\frac{m}{\sqrt{q}}\right)
\mathop{\sum_{n\geq 1\atop (n,q)=1}}\frac{\lambda_f(n)}{n^{s_0}}W_{s_0}\left(\frac{n}{q}\right)\notag \\
&\ll&\sum_{m\leq(q\tau)^{\frac{1}{2}+\varepsilon}}\frac{1}{m^{\sigma_0}}
\sum_{n\leq(q\tau)^{1+\varepsilon}}\frac{|\lambda_f(n)|}{n^{\sigma_0}}
\ll(q\tau)^{\frac{3}{2}-\frac{3}{2}\sigma_0+\theta+\varepsilon}.
\eea
Next, we evaluate $S_{11}$  which contributes the main term. We write
\bna
S_{11}&=& S_{11}^{\ast}+S_{11}^{\ast\ast},
\ena
where
\bna
S_{11}^{\ast}&=&\sum_{\pm}\sum_{m\geq 1\atop (m,q)=1}\frac{1}{m^{\overline{s_0}}}V_{\overline{s_0}}\left(\frac{m}{\sqrt{q}}\right)
\mathop{\sum_{n\geq 1 \atop n\equiv\pm m\text(\bmod q)}}_{ n\neq \pm m}\frac{\lambda_f(n)}{n^{s_0}}W_{s_0}\left(\frac{n}{q}\right), \\
S_{11}^{\ast\ast}&=&\sum_{\pm}\sum_{m\geq 1\atop (m,q)=1}\frac{1}{m^{\overline{s_0}}}V_{\overline{s_0}}\left(\frac{m}{\sqrt{q}}\right)
\mathop{\sum_{n\geq 1\atop n=\pm m}}\frac{\lambda_f(n)}{n^{s_0}}W_{s_0}\left(\frac{n}{q}\right).
\ena
We have
\bna
S_{11}^{\ast}\ll(q\tau)^{\theta+\varepsilon}\sum_{\pm}\sum_{m\leq(q\tau)^{\frac{1}{2}+\varepsilon} }\frac{1}{m^{\sigma_0}}
\mathop{\sum_{n\leq(q\tau)^{1+\varepsilon}\atop n\equiv\pm m\text(\bmod q)}}_{n \neq\pm m}\frac{1}{n^{\sigma_0}},
\ena
where the contribution from $n< m$ is at most
\bna
&&(q\tau)^{\theta+\varepsilon}\sum_{\pm}\sum_{m\leq(q\tau)^{\frac{1}{2}+\varepsilon} }\frac{1}{m^{\sigma_0}}
\mathop{\sum_{n\leq(q\tau)^{1+\varepsilon}\atop n\equiv\pm m\text(\bmod q)}}_{n < m}\frac{1}{n^{\sigma_0}} \notag\\
&\ll&(q\tau)^{\theta+\varepsilon}\sum_{\pm}\sum_{n\leq(q\tau)^{\frac{1}{2}+\varepsilon}}\frac{1}{n^{\sigma_0}}
\sum_{m\leq(q\tau)^{\frac{1}{2}+\varepsilon} \atop n\equiv\pm m\text(\bmod q) }\frac{1}{m^{\sigma_0}}\notag\notag \\
&\ll&(q\tau)^{\theta+\varepsilon}\sum_{\pm}\sum_{n\leq(q\tau)^{\frac{1}{2}+\varepsilon}}\frac{1}{n^{\sigma_0}}\sum_{1\leq k\leq q^{-1}(q\tau)^{1/2+\varepsilon}}\frac{1}{(qk)^{\sigma_0}}\notag\\
&\ll&q^{-1}(q\tau)^{1-\sigma_0+\theta+\varepsilon},
\ena
and the contribution from $n>m$ is bounded by
\bna
&&(q\tau)^{\theta+\varepsilon}\sum_{\pm}\sum_{m\leq(q\tau)^{\frac{1}{2}+\varepsilon} }\frac{1}{m^{\sigma_0}}
\mathop{\sum_{n\leq(q\tau)^{1+\varepsilon}\atop n\equiv\pm m\text(\bmod q)}}_{n >m}\frac{1}{n^{\sigma_0}}\notag\\
&\ll&(q\tau)^{\theta+\varepsilon}\sum_{m\leq(q\tau)^{\frac{1}{2}+\varepsilon}}\frac{1}{m^{\sigma_0}}\sum_{k<q^{-1}(q\tau)^{1+\varepsilon}}\frac{1}{(qk)^{\sigma_0}}\notag\\
&\ll& q^{-1}(q\tau)^{\frac{3}{2}-\frac{3}{2}\sigma_0+\theta+\varepsilon}.
\ena
Hence
\bea\label{S_11*}
S_{11}^{\ast}\ll q^{-1}(q\tau)^{\frac{3}{2}(1-\sigma_0)+\theta+\varepsilon}.
\eea
Note that
\bna
S_{11}^{\ast\ast}=\sum_{(n,q)=1}\frac{\lambda_f(n)}{n^{2\sigma_0}}W_{s_0}\left(\frac{n}{q}\right)
V_{\overline{s_0}}\left(\frac{n}{\sqrt{q}}\right),
\ena
which has been evaluated in Sono \cite{sono} (see Pages 1130-1131 in \cite{sono}). By Sono \cite{sono}, we have
\bea\label{S_11**}
S_{11}^{\ast\ast}=L(2\sigma_0,f)+O(q^{-2\sigma_0+\varepsilon}\tau^{\theta+\varepsilon}
+q^{-\frac{3}{4}+\varepsilon}\tau^{\frac{3}{2}(1-\sigma_0)}).
\eea
By (\ref{S_11*}) and (\ref{S_11**}), we obtain
\bea\label{S_11}
S_{11}= L(2\sigma_0,f)+O(q^{-\frac{3}{4}+\varepsilon}\tau^{\frac{3}{2}(1-\sigma_0)}+
q^{-1}(q\tau)^{\frac{3}{2}(1-\sigma_0)+\theta+\varepsilon}+q^{-2\sigma_0+\varepsilon}\tau^{\theta+\varepsilon}).
\eea
Therefore, by (\ref{S_1}), (\ref{S_12}) and (\ref{S_11}),
\bna
S_1=\frac{q}{2}L(2\sigma_0,f)+O\left(q^{1-2\sigma_0-\theta+\varepsilon}(q\tau)^{\theta+\varepsilon}+
q^{\frac{1}{4}+\varepsilon}\tau^{\frac{3}{2}(1-\sigma_0)}+
(q\tau)^{\frac{3}{2}(1-\sigma_0)+\theta+\varepsilon}\right).
\ena
Note that the first term is dominated by the third term. Then Lemma \ref{jmh} follows.
\end{proof}
\section{Estimation of $S_2$}
Recall that
\bea\label{S_2}
S_2=q^{-\overline{s_0}}\frac{\gamma(1-\overline{s_0})}{\gamma(\overline{s_0})}\sum_{m=1}^{\infty}\frac{1}{m^{1-\overline{s_0}}}V_{1-\overline{s_0}}\left(\frac{m}{\sqrt{q}}\right)
\sum_{n=1}^{\infty}\frac{\lambda_f(n)}{n^{s_0}}W_{s_0}\left(\frac{n}{q}\right)\sideset{}{^\dagger}\sum_{\chi \text(\bmod q) \atop \chi(-1)=1}\chi(m)\chi(n)\tau(\overline{\chi}).
\eea
\begin{lemma}\label{jmhh}
For any $\varepsilon >0$ , we have
\bna\label{lemm2}
S_2\ll q^{\frac{1}{2}}(q\tau)^{-\frac{1}{2}\sigma_0+\frac{1}{2}+\varepsilon}
+q^{-\frac{1}{2}}(q\tau)^{\frac{3}{2}-\frac{3}{2}\sigma_0+\theta+\varepsilon}.
\ena
\end{lemma}
\begin{proof}
By (\ref{S_2}), we have
\bea\label{b}
S_2
&=&\frac{1}{2}q^{-\overline{s_0}}\frac{\gamma(1-\overline{s_0})}{\gamma(\overline{s_0})}\sum_{m=1}^{\infty}\frac{1}{m^{1-\overline{s_0}}}V_{1-\overline{s_0}}\left(\frac{m}{\sqrt{q}}\right)
\sum_{n=1}^{\infty}\frac{\lambda_f(n)}{n^{s_0}}W_{s_0}\left(\frac{n}{q}\right)\sideset{}{^\dagger}\sum_{\chi \text(\bmod q)}\left[1+\chi(-1)\right]\chi(mn)\tau(\overline{\chi})\notag \\
&=&\frac{1}{2}q^{-\overline{s_0}}\frac{\gamma(1-\overline{s_0})}{\gamma(\overline{s_0})}\sum_{\pm}\sum_{m=1}^{\infty}\frac{1}{m^{1-\overline{s_0}}}V_{1-\overline{s_0}}\left(\frac{m}{\sqrt{q}}\right)
\sum_{n=1}^{\infty}\frac{\lambda_f(n)}{n^{s_0}}W_{s_0}\left(\frac{n}{q}\right)\sideset{}{^\dagger}\sum_{\chi \text(\bmod q)}\chi(\pm mn)\tau(\overline{\chi}).
\eea
By the orthogonality of Dirichlet characters, we have
\bea\label{a}
\sideset{}{^\dagger}\sum_{\chi \text(\bmod q)}\chi(\pm m n)\tau(\overline{\chi})
&=&\sideset{}{^\dagger}\sum_{\chi \text(\bmod q)}\ \sideset{}{^\ast}\sum_{a \text(\bmod q)}\overline{\chi}(a)e\left(\frac{a}{q}\right)\chi(\pm mn)\notag \\
&=&\sideset{}{^\ast}\sum_{a \text(\bmod q)}e\left(\frac{a}{q}\right)\sideset{}{^\dagger}\sum_{\chi \text(\bmod q)}\overline{\chi}(a)\chi(\pm mn)\notag \\
&=&\sideset{}{^\ast}\sum_{a \text(\bmod q)}e\left(\frac{a}{q}\right)\left[\phi(q)\mathbf{1}_{a\equiv\pm nm (\bmod q)}-1\right]\notag \\
&=&\phi(q)e\left(\frac{\pm nm}{q}\right)+1.
\eea
Plugging (\ref{a}) into (\ref{b}), one has
\bea
S_2=\frac{\phi(q)}{2}S_{21}+S_{22},
\eea
where
\bna
S_{21}&=&q^{-\overline{s_0}}\frac{\gamma(1-\overline{s_0})}{\gamma(\overline{s_0})}\sum_{\pm}\sum_{m \geq1 \atop (m,q)=1}\frac{1}{m^{1-\overline{s_0}}}V_{1-\overline{s_0}}\left(\frac{m}{\sqrt{q}}\right)
\sum_{n\geq 1 \atop (n,q)=1}\frac{\lambda_f(n)}{n^{s_0}}W_{s_0}\left(\frac{n}{q}\right)e\left(\frac{\pm nm}{q}\right),\\
S_{22}&=&q^{-\overline{s_0}}\frac{\gamma(1-\overline{s_0})}{\gamma(\overline{s_0})}\sum_{m \geq1 \atop (m,q)=1}\frac{1}{m^{1-\overline{s_0}}}V_{1-\overline{s_0}}\left(\frac{m}{\sqrt{q}}\right)
\sum_{n\geq 1 \atop (n,q)=1}\frac{\lambda_f(n)}{n^{s_0}}W_{s_0}\left(\frac{n}{q}\right).
\ena
Trivially,
\bea
S_{22}&\ll& q^{-\sigma_0}\tau^{\frac{1}{2}-\sigma_0}\sum_{m\leq(q\tau)^{\frac{1}{2}+\varepsilon}}\frac{1}{m^{1-\sigma_0}}
\sum_{n\leq(q\tau)^{1+\varepsilon}}\frac{|\lambda_f(n)|}{n^{\sigma_0}}\notag \\
&\ll&q^{-\sigma_0}\tau^{\frac{1}{2}-\sigma_0}(q\tau)^{\frac{1}{2}\sigma_0+\varepsilon}
(q\tau)^{1-\sigma_0+\theta+\varepsilon}\notag \\
&\ll& q^{-\frac{1}{2}}(q\tau)^{\frac{3}{2}-\frac{3}{2}\sigma_0+\theta+\varepsilon}.
\eea

As for $S_{21}$, removing the condition $(n, q) = 1$, we have
\bea
S_{21}&=&q^{-\overline{s_0}}\frac{\gamma(1-\overline{s_0})}{\gamma(\overline{s_0})}\sum_{\pm}\sum_{m \geq1 \atop (m,q)=1}\frac{1}{m^{1-\overline{s_0}}}V_{1-\overline{s_0}}\left(\frac{m}{\sqrt{q}}\right)
\sum_{n\geq 1}\frac{\lambda_f(n)}{n^{s_0}}W_{s_0}\left(\frac{n}{q}\right)e\left(\frac{\pm nm}{q}\right)\notag\\
&+&O\left(q^{-\frac{3}{2}}(q\tau)^{\frac{3}{2}-\frac{3}{2}\sigma_0+\theta+\varepsilon}\right).
\eea
By Lemma \ref{e} and partial summation, we have
\bna
&&\sum_{n=1}^{\infty}\frac{\lambda_f(n)}{n^{s_0}}W_{s_0}\left(\frac{n}{q}\right)e\left(\frac{\pm nm}{q}\right)\\
&=&\sum_{n\leq (q\tau)^{1+\varepsilon}}\frac{\lambda_f(n)}{n^{s_0}}W_{s_0}\left(\frac{n}{q}\right)e\left(\frac{\pm nm}{q}\right)+O((q\tau)^{-K})\\
&=&\sum_{\alpha \atop 2^{\alpha}=N\ll (q\tau)^{1+\varepsilon}}\sum_{n}\lambda_f(n)e\left(\frac{\pm nm}{q}\right)
\Phi\left(\frac{n}{N}\right)n^{-s_0}W_{s_0}\left(\frac{n}{q}\right)+O((q\tau)^{-K})\\
&=&\sum_{\alpha \atop 2^{\alpha}=N\ll (q\tau)^{1+\varepsilon}}\int_{1}^{\infty}\Phi\left(\frac{t}{N}\right)t^{-s_0}W_{s_0}\left(\frac{t}{q}\right)\mathrm{d}\left(\sum_{n\leq t}e\left(\frac{\pm nm}{q}\right)\lambda_f(n)\right)+O((q\tau)^{-K})\\
&=&\sum_{\alpha \atop 2^{\alpha}=N\ll (q\tau)^{1+\varepsilon}}\int_{1}^{\infty}\sum_{n\leq t}\lambda_f(n)e\left(\frac{\pm nm}{q}\right)\left(\Phi\left(\frac{t}{N}\right)t^{-s_0}W_{s_0}\left(\frac{t}{q}\right)\right)^{\prime}\mathrm{d}t+O((q\tau)^{-K}),\\
\ena
where $K$ denotes arbitrary large number, $\Phi(x)$ is a smooth function compactly supported on $[1,2]$ and
\bna
\frac{\mathrm{d}}{\mathrm{d}t}\left(\Phi\left(\frac{t}{N}\right)t^{-s_0}W_{s_0}\left(\frac{t}{q}\right)\right)\ll N^{-\sigma_0-1}.
\ena
Hence, for $\sigma_0\geq\frac{1}{2}$,
\bna
\sum_{n=1}^{\infty}\frac{\lambda_f(n)}{n^{s_0}}W_{s_0}\left(\frac{n}{q}\right)e\left(\frac{\pm nm}{q}\right)\ll\sum_{\alpha \atop 2^{\alpha}=N\ll (q\tau)^{1+\varepsilon}} N\cdot N^{\frac{1}{2}+\varepsilon}\cdot N^{-\sigma_0-1}\ll (q\tau)^{\varepsilon}.
\ena
Thus
\bea
S_{21}&\ll&
q^{-\sigma_0}\tau^{\frac{1}{2}-\sigma_0}(q\tau)^{\frac{1}{2}\sigma_0+\varepsilon}(q\tau)^{\varepsilon}
\ll q^{-\frac{1}{2}}(q\tau)^{\frac{1}{2}-\frac{1}{2}\sigma_0+\varepsilon}.
\eea
By (5.4)-(5.7),
\bna
S_2&\ll& q^{\frac{1}{2}}(q\tau)^{-\frac{1}{2}\sigma_0+\frac{1}{2}+\varepsilon}
+q^{-\frac{1}{2}}(q\tau)^{\frac{3}{2}-\frac{3}{2}\sigma_0+\theta+\varepsilon}.
\ena
\end{proof}
\section{Estimation of $S_3$}
Recall that
\bea\label{S_3}
S_3=q^{-2s_0-\overline{s_0}}\frac{\gamma(1-\overline{s_0})\widetilde{\gamma}(1-s_0)}{\gamma(\overline{s_0})\widetilde{\gamma}(s_0)}
\sum_{m=1}^{\infty}\frac{1}{m^{1-\overline{s_0}}}V_{1-\overline{s_0}}\left(\frac{m}{\sqrt{q}}\right)
\sum_{n=1}^{\infty}\frac{\lambda_f(n)}{n^{1-s_0}}W_{1-s_0}\left(\frac{n}{q}\right)\notag\\
\times\sideset{}{^\dagger}\sum_{\chi \text(\bmod q) \atop \chi(-1)=1}\chi(m)\overline{\chi}(n)\tau(\overline{\chi})\tau(\chi)^{2}.
\eea
\begin{lemma}\label{jmhhh}
For any $\varepsilon>0$, we have
\bna\label{lemm3}
S_3\ll q^{\frac{1}{2}}(q\tau)^{1-\frac{3}{2}\sigma_0+\varepsilon}
+q^{-\frac{1}{2}}(q\tau)^{\frac{3}{2}-\frac{3}{2}\sigma_0+\theta+\varepsilon}.
\ena
\end{lemma}
\begin{proof}
By (\ref{S_3}) and (\ref{q_1}) in Lemma \ref{q},
\bna
S_3
=\frac{1}{2}q^{1-2s_0-\overline{s_0}}\frac{\gamma(1-\overline{s_0})\widetilde{\gamma}(1-s_0)}{\gamma(\overline{s_0})\widetilde{\gamma}(s_0)}
\sum_{\pm}\sum_{m=1}^{\infty}\frac{1}{m^{1-\overline{s_0}}}V_{1-\overline{s_0}}\left(\frac{m}{\sqrt{q}}\right)
\sum_{n=1}^{\infty}\frac{\lambda_f(n)}{n^{1-s_0}}W_{1-s_0}\left(\frac{n}{q}\right)\\
\times\sideset{}{^\dagger}\sum_{\chi \text(\bmod q)}\chi(\pm m\overline{n})\tau(\chi).
\ena
By the orthogonality of Dirichlet characters, we have
\bna
\sideset{}{^\dagger}\sum_{\chi \text(\bmod q)}\chi(\pm m\overline{n})\tau(\chi)
&=&\sideset{}{^\ast}\sum_{a \text(\bmod q)}e\left(\frac{a}{q}\right)\sideset{}{^\dagger}\sum_{\chi\text(\bmod q)}\chi(\pm am\overline{n})\notag \\
&=&\sideset{}{^\ast}\sum_{a \text(\bmod q)}e\left(\frac{a}{q}\right)\left[\phi(q)\mathbf{1}_{a\equiv\pm n\overline{m} (\bmod q)}-1\right]\notag \\
&=&\phi(q)e\left(\frac{\pm n\overline{m}}{q}\right)+1.
\ena
Thus
\bea
S_3&=&\frac{\phi(q)}{2}q^{1-2s_0-\overline{s_0}}\frac{\gamma(1-\overline{s_0})\widetilde{\gamma}(1-s_0)}{\gamma(\overline{s_0})\widetilde{\gamma}(s_0)}
\sum_{\pm}\sum_{m\geq1 \atop (m,q)=1}\frac{1}{m^{1-\overline{s_0}}}V_{1-\overline{s_0}}\left(\frac{m}{\sqrt{q}}\right)
\sum_{n\geq1 \atop (n,q)=1}\frac{\lambda_f(n)}{n^{1-s_0}}W_{1-s_0}\left(\frac{n}{q}\right)e\left(\frac{\pm n\overline{m}}{q}\right)\notag \\
&+&q^{1-2s_0-\overline{s_0}}\frac{\gamma(1-\overline{s_0})\widetilde{\gamma}(1-s_0)}{\gamma(\overline{s_0})\widetilde{\gamma}(s_0)}
\sum_{m\geq1 \atop (m,q)=1}\frac{1}{m^{1-\overline{s_0}}}V_{1-\overline{s_0}}\left(\frac{m}{\sqrt{q}}\right)
\sum_{n\geq1 \atop (n,q)=1}\frac{\lambda_f(n)}{n^{1-s_0}}W_{1-s_0}\left(\frac{n}{q}\right)\notag \\
&:=&\frac{\phi(q)}{2}q^{1-2s_0-\overline{s_0}}\frac{\gamma(1-\overline{s_0})\widetilde{\gamma}(1-s_0)}{\gamma(\overline{s_0})\widetilde{\gamma}(s_0)}S_{31}
+q^{1-2s_0-\overline{s_0}}\frac{\gamma(1-\overline{s_0})\widetilde{\gamma}(1-s_0)}{\gamma(\overline{s_0})\widetilde{\gamma}(s_0)}S_{32},
\eea
say. Trivially, we have
\bea
S_{32}\ll\sum_{m\leq(q\tau)^{\frac{1}{2}+\varepsilon}}m^{\sigma_0-1}\sum_{n\leq(q\tau)^{1+\varepsilon}}\frac{|\lambda_f(n)|}{n^{1-\sigma_0}}
\ll(q\tau)^{\frac{3}{2}\sigma_0+\theta+\varepsilon}.
\eea

Removing the condition $(n,q)=1$, we have
\bea
S_{31}=\sum_{\pm}\sum_{m\geq1 \atop (m,q)=1}\frac{1}{m^{1-\overline{s_0}}}V_{1-\overline{s_0}}\left(\frac{m}{\sqrt{q}}\right)
\sum_{n\geq1}\frac{\lambda_f(n)}{n^{1-s_0}}W_{1-s_0}\left(\frac{n}{q}\right)e\left(\frac{\pm n\overline{m}}{q}\right)+O(q^{-1}(q\tau)^{\frac{3}{2}\sigma_0+\theta+\varepsilon}).
\eea
By partial integration once and Lemma \ref{e}, we obtain
\bea
&&
\sum_{\pm}\sum_{m\geq1 \atop (m,q)=1}\frac{1}{m^{1-\overline{s_0}}}V_{1-\overline{s_0}}\left(\frac{m}{\sqrt{q}}\right)
\sum_{n\geq1}\frac{\lambda_f(n)}{n^{1-s_0}}W_{1-s_0}\left(\frac{n}{q}\right)e\left(\frac{\pm n\overline{m}}{q}\right)\notag\\
&\ll&\sum_{\pm}\sum_{\alpha \atop 2^{\alpha}=N\leq(q\tau)^{1+\varepsilon}}\sum_{m\leq(q\tau)^{1/2+\varepsilon}}m^{\sigma_0-1}\
\left|\sum_{n\geq1}\frac{\lambda_f(n)}{n^{1-s_0}}W_{1-s_0}\left(\frac{n}{q}\right)e\left(\frac{\pm n\overline{m}}{q}\right)\Phi\left(\frac{n}{N}\right)\right|\notag\\
&\ll&\sum_{\pm}\sum_{\alpha \atop 2^{\alpha}=N\leq(q\tau)^{1+\varepsilon}}\sum_{m\leq(q\tau)^{1/2+\varepsilon}}m^{\sigma_0-1}
\left|\int_{0}^{\infty}t^{s_0-1}W_{1-s_0}\left(\frac{t}{q}\right)\Phi\left(\frac{t}{N}\right)\mathrm{d}\left(\sum_{n\leq t}\lambda_f(n)e\left(\frac{\pm n\overline{m}}{q}\right)\right)\right|\notag\\
&=&\sum_{\pm}\sum_{\alpha \atop 2^{\alpha}=N\leq(q\tau)^{1+\varepsilon}}\sum_{m\leq(q\tau)^{1/2+\varepsilon}}m^{\sigma_0-1}
\left|\int_{0}^{\infty}\sum_{n\leq t}\lambda_f(n)e\left(\frac{\pm n\overline{m}}{q}\right)\left(t^{s_0-1}W_{1-s_0}\left(\frac{t}{q}\right)\Phi\left(\frac{t}{N}\right)\right)^{\prime}\mathrm{d}t\right|\notag\\
&\ll&(q\tau)^{\frac{1}{2}\sigma_0+\varepsilon}\sum_{\alpha \atop 2^{\alpha}=N\leq(q\tau)^{1+\varepsilon}}N\cdot N^{\frac{1}{2}+\varepsilon}N^{\sigma_0-2}\notag\\
&\ll& (q\tau)^{\frac{3}{2}\sigma_0-\frac{1}{2}+\varepsilon}.
\eea
By applying (5.113) in Iwaniec and Kowalski \cite{ik}, we have
\bea
\frac{\gamma(1-\overline{s_0})\widetilde{\gamma}(1-s_0)}{\gamma(\overline{s_0})\widetilde{\gamma}(s_0)}
=O(\tau^{\frac{3}{2}-3\sigma_0}).
\eea
By (6.2)-(6.6), we conclude that
\bna
S_3&\ll& q^{2-3\sigma_0}\tau^{\frac{3}{2}-3\sigma_0}\big((q\tau)^{\frac{3}{2}\sigma_0-\frac{1}{2}+\varepsilon}
+q^{-1}(q\tau)^{\frac{3}{2}\sigma_0+\theta+\varepsilon}\big)+q^{1-3\sigma_0}\tau^{\frac{3}{2}-3\sigma_0}
(q\tau)^{\frac{3}{2}\sigma_0+\theta+\varepsilon}\notag \\
&\ll& q^{\frac{1}{2}}(q\tau)^{1-\frac{3}{2}\sigma_0+\varepsilon}
+q^{-\frac{1}{2}}(q\tau)^{\frac{3}{2}-\frac{3}{2}\sigma_0+\theta+\varepsilon}.
\ena
\end{proof}
\section{Estimation of $S_4$}
Recall that
\bna
S_4=q^{-2s_0}\frac{\widetilde{\gamma}(1-s_0)}{\widetilde{\gamma}(s_0)}
\sum_{m=1}^{\infty}\frac{1}{m^{\overline{s_0}}}V_{\overline{s_0}}\left(\frac{m}{\sqrt{q}}\right)
\sum_{n=1}^{\infty}\frac{\lambda_f(n)}{n^{1-s_0}}W_{1-s_0}\left(\frac{n}{q}\right)
\sideset{}{^\dagger}\sum_{\chi \text(\bmod q) \atop \chi(-1)=1}\overline{\chi}(m)\overline{\chi}(n)\tau(\chi)^{2}.
\ena
By the orthogonality of Dirichlet characters, we have
\bea\label{x}
\sideset{}{^\dagger}\sum_{\chi \text(\bmod q) \atop \chi(-1)=1}\overline{\chi}(m)\overline{\chi}(n)\tau(\chi)^{2}
&=&\frac{1}{2}\ \sideset{}{^\dagger}\sum_{\chi \text(\bmod q)}\left(\chi(-1)+1\right)\overline{\chi}(mn)\tau(\chi)^{2}\notag
=\frac{1}{2}\sum_{\pm}\sideset{}{^\dagger}\sum_{\chi \text(\bmod q)}\overline{\chi}(\pm mn)\tau(\chi)^{2}\notag\\
&=&\frac{1}{2}\sum_{\pm}\sideset{}{^\dagger}\sum_{\chi \text(\bmod q)}\overline{\chi}(\pm mn)
\left(\sideset{}{^\ast}\sum_{a  \bmod q}\chi(a)e\left(\frac{a}{q}\right)\right)^{2}\notag\\
&=&\frac{1}{2}\sum_{\pm}\sideset{}{^\ast}\sum_{a  \bmod q} \ \sideset{}{^\ast}\sum_{b  \bmod q }e\left(\frac{a+b}{q}\right)\left[\phi(q)\mathbf{1}_{ab\equiv\pm nm (\bmod q)}-1\right]\notag\\
&=&\frac{1}{2}\phi(q)\sum_{\pm}\sideset{}{^\ast}\sum_{a  \bmod q} \ \sideset{}{^\ast}\sum_{b \bmod q\atop b\equiv \pm \overline{a}mn \bmod q}e\left(\frac{a+b}{q}\right)-\frac{1}{2}\sum_{\pm}\sideset{}{^\ast}\sum_{a  \bmod q} \ \sideset{}{^\ast}\sum_{b  \bmod q}e\left(\frac{a+b}{q}\right)\notag\\
&=&\frac{1}{2}\phi(q)\sum_{\pm}S(1,\pm{mn};q)-1.
\eea

By (\ref{x}), we have
\bea\label{S_4}
S_4=\frac{1}{2}\phi(q)q^{-2s_0}\frac{\widetilde{\gamma}(1-s_0)}{\widetilde{\gamma}(s_0)}S_{41}
-q^{-2s_0}\frac{\widetilde{\gamma}(1-s_0)}{\widetilde{\gamma}(s_0)}S_{42},
\eea
where
\bna
S_{41}&=&\sum_{\pm}\sum_{m\geq 1 \atop (m,q)=1}\frac{1}{m^{\overline{s_0}}}V_{\overline{s_0}}\left(\frac{m}{\sqrt{q}}\right)
\sum_{n\geq 1 \atop (n,q)=1}\frac{\lambda_f(n)}{n^{1-s_0}}W_{1-s_0}\left(\frac{n}{q}\right)S(1,\pm{mn};q),\\
S_{42}&=&\sum_{m\geq 1 \atop (m,q)=1}\frac{1}{m^{\overline{s_0}}}V_{\overline{s_0}}\left(\frac{m}{\sqrt{q}}\right)
\sum_{n\geq 1 \atop (n,q)=1}\frac{\lambda_f(n)}{n^{1-s_0}}W_{1-s_0}\left(\frac{n}{q}\right).
\ena
say. By Weil's bound for Kloosterman sums and Lemma \ref{1}, we have
\bea\label{S_42}
S_{42}\ll\sum_{m\leq (q\tau)^{1/2+\varepsilon}}m^{-\sigma_0}
\sum_{n\leq (q\tau)^{1+\varepsilon}}\frac{|\lambda_f(n)|}{n^{1-\sigma_0}}
\ll(q\tau)^{\frac{1}{2}-\frac{1}{2}\sigma_0+\varepsilon}(q\tau)^{\sigma_0+\theta+\varepsilon}
\ll(q\tau)^{\frac{1}{2}+\frac{1}{2}\sigma_0+\theta+\varepsilon}.
\eea
Making smooth partitions of unity into dyadic segments to the sums over $m$ and $n$, we arrive at
\bea\label{S_41}
S_{41}
&=&\sum_{\pm}\sum_{\alpha_1 \atop 2^{\alpha_{1}}=M\ll(q\tau)^{1/2+\varepsilon}}\sum_{\alpha_{2} \atop 2^{\alpha_2}=N\ll(q\tau)^{1+\varepsilon}}T(\pm,M,N),
\eea
where
\bea\label{T}
T(\pm,M,N)=\sum_{ (m,q)=1}m^{-\overline{s_0}}V_{\overline{s_0}}\left(\frac{m}{\sqrt{q}}\right)\mathcal{W}_1\left(\frac{m}{M}\right)\sum_{(n,q)=1}\frac{\lambda_f(n)}{n^{1-s_0}}
W_{1-s_0}\left(\frac{n}{q}\right)\mathcal{W}_2\left(\frac{n}{N}\right)S(1,\pm mn;q).
\eea
Here $\mathcal{W}_{j}(x)\in C_c^{\infty}(1,2) \ (j=1,2)$ satisfying $\mathcal{W}_{j}^{(l)}(x)\ll_l 1$ for  $l\geq0$.
\begin{lemma}\label{last lemma}
For any $\varepsilon>0$, we have
\bna
T(\pm,M,N)&\ll& \tau^{-\frac{1}{2}}(q\tau)^{\frac{1}{2}+\beta_{1}-(\beta_{1}-\beta_{2})\sigma_0+\varepsilon}
+\tau(q\tau)^{(1-\beta_1)\sigma_0+\frac{1}{2}+\varepsilon}\notag \\
&+&\tau(q\tau)^{1+(1-\sigma_0)\beta_1-(1-\sigma_0)\beta_2+(2-\beta_2)\theta+\varepsilon}.
\ena
\end{lemma}
\begin{proof}
We distinguish three cases according to the ranges of $M$ and $N$. Let $\beta_1$ and $\beta_2$ be positive parameters to be chosen later. By (\ref{S_41}), we assume $0<\beta_1<\frac{1}{2}$ and $0<\beta_2<1$. \\
{\bf Case I:} $M <(q\tau)^{\beta_{1}}, N <(q\tau)^{\beta_{2}}$.

By Lemma \ref{1} and Weil's bound for Kloosterman sums, we have
\bea\label{case1}
T(\pm,M,N)
&\ll& q^{\frac{1}{2}}\sum_{m \asymp M}m^{-\sigma_0}
\sum_{n\asymp N}n^{\sigma_0-1}\notag \\
&\ll&q^{\frac{1}{2}}N^{\sigma_0+\varepsilon}M^{1-\sigma_0}
\ll q^{\frac{1}{2}}(q\tau)^{\beta_{1}-(\beta_{1}-\beta_{2})\sigma_0+\varepsilon}.
\eea
{\bf Case II:} $M \geq (q\tau)^{\beta_{1}} $.

In this case, we apply Poisson summation formula to the $m$-sum to get
\bea\label{2}
&&\sum_{ (m,q)=1}\frac{1}{m^{\overline{s_0}}}V_{\overline{s_0}}\left(\frac{m}{\sqrt{q}}\right)\mathcal{W}_1\left(\frac{m}{M}\right)
e\left(\frac{\pm \overline{a}mn}{q}\right)\notag\\
&=&\frac{1}{M^{\overline{s_0}}}\sideset{}{^\ast}\sum_{\beta \bmod q}
e\left(\frac{\pm \overline{a}\beta n}{q}\right)\sum_{m\equiv \beta  (\bmod q)}\left(\frac{M}{m}\right)^{\overline{s_0}}
V_{\overline{s_0}}\left(\frac{m}{\sqrt{q}}\right)\mathcal{W}_1\left(\frac{m}{M}\right)\notag \\
&=&\frac{1}{M^{\overline{s_0}}}\sideset{}{^\ast}\sum_{\beta  \bmod q}e\left(\frac{\pm \overline{a}\beta n}{q}\right) \ \frac{1}{q}\sum_{m\in\mathbb{Z}}e\left(\frac{\beta m}{q}\right)\int_{\mathbb{R}}\left(\frac{M}{t}\right)^{\overline{s_0}}
V_{\overline{s_0}}\left(\frac{t}{\sqrt{q}}\right)\mathcal{W}_1\left(\frac{t}{M}\right)e\left(\frac{-m t}{q}\right)\mathrm{d}t\notag\\
&=&{q}^{-1}M^{1-\overline{s_0}}\sum_{m\in\mathbb{Z}} \ \sideset{}{^\ast}\sum_{\beta \bmod q}e\left(\frac{(m\pm \overline{a} n)\beta}{q}\right)J(m,q),
\eea
where
\bna
J(m,q)=\int_{\mathbb{R}}t^{-\overline{s_0}}
V_{\overline{s_0}}\left(\frac{M t}{\sqrt{q}}\right)\mathcal{W}_1(t)e\left(\frac{- M m t}{q}\right)\mathrm{d}t.
\ena
By repeated partial integrations,
\bna
J(m,q)\ll \left(1+\frac{|m|M}{q\tau}\right)^{-A}
\ena
for any $A>0$. Thus the contribution from $|m|\geq(q\tau)^{1+\varepsilon}/M$ can be arbitrarily small.

Plugging (\ref{2}) into (\ref{T}), one has
\bea
T(\pm,M,N)&=&{q}^{-1}M^{1-\overline{s_0}}\sum_{|m|\leq (q\tau)^{1+\varepsilon}/M}J(m,q)\sum_{ (n,q)=1}\frac{\lambda_f(n)}{n^{1-s_0}}W_{1-s_0}\left(\frac{n}{q}\right)\mathcal{W}_2\left(\frac{n}{N}\right)
\notag\\
&\times&\sideset{}{^\ast}\sum_{a \bmod q}e\left(\frac{a}{q}\right)  \sideset{}{^\ast}\sum_{\beta  \bmod q}e\left(\frac{(m\pm \overline{a} n)\beta}{q} \right)+O\left((q\tau)^{-A}\right)\notag\\
&:=&{q}^{-1}M^{1-\overline{s_0}}\sum_{|m|\leq (q\tau)^{1+\varepsilon}/M}J(m,q)\sum_{ (n,q)=1}\frac{\lambda_f(n)}{n^{1-s_0}}W_{1-s_0}\left(\frac{n}{q}\right)\mathcal{W}_2\left(\frac{n}{N}\right)c(m,n;q)\notag\\
&+&O\left((q\tau)^{-A}\right),\\ \notag
\eea
where
\bna
c(m,n;q)=\sideset{}{^\ast}\sum_{a \bmod q}e\left(\frac{a}{q}\right)\sideset{}{^\ast}\sum_{\beta \bmod q}e\left(\frac{(m\pm \overline{a} n)\beta}{q}\right).
\ena
$1^{\circ}$ If $m=0$, then
\bna
c(0,n;q)=\sideset{}{^\ast}\sum_{a  \bmod q}e\left(\frac{ a}{q}\right)\sideset{}{^\ast}\sum_{\beta  \bmod q}e\left(\frac{\pm \overline{a}n\beta}{q}\right)=
\left\{\begin{array}{ll}
-\phi(q),&\mbox{if}\, \ q\mid n,\\
1,& \mbox{if}\, \ q\nmid n.
\end{array}\right.
\ena
Thus the contribution from the terms with $m=0$ is at most
\bea
&&\frac{M^{1-\sigma_0}}{q}|J(0,q)|\sum_{n\asymp N \atop q\mid n}\frac{|\lambda_f(n)|}{n^{1-\sigma_0}}\phi(q)
+\frac{M^{1-\sigma_0}}{q}|J(0,q)|\sum_{n\asymp N \atop q\nmid n}\frac{|\lambda_f(n)|}{n^{1-\sigma_0}}\notag\\
&\ll&q^{-1}M^{1-\sigma_0}N^{\sigma_0+\theta+\varepsilon}.
\eea
$2^{\circ}$ If $m\neq0$, then
\bna
c(m,n;q)
&=&\sideset{}{^\ast}\sum_{\beta \bmod q}e\left(\frac{a}{q}\right)\left(\sum_{\beta \bmod q}e\left(\frac{(m\pm \overline{a} n)\beta}{q}\right)-1\right)\\
&=&qe\left(\frac{\mp n\overline{m}}{q}\right)+1.
\ena
Thus the contribution from the terms with $m\neq0$ is
\bea\label{mnot0}
&&M^{1-\overline{s_0}}\sum_{ (n,q)=1}\frac{\lambda_f(n)}{n^{1-s_0}}W_{1-s_0}\left(\frac{n}{q}\right)\mathcal{W}_2\left(\frac{n}{N}\right)
\sum_{0<|m|\leq (q\tau)^{1+\varepsilon}/M}e\left(\frac{\mp n\overline{m}}{q}\right)J(m,q)\notag\\
&+&q^{-1}M^{1-\overline{s_0}}\sum_{ (n,q)=1}\frac{\lambda_f(n)}{n^{1-s_0}}W_{1-s_0}\left(\frac{n}{q}\right)\mathcal{W}_2\left(\frac{n}{N}\right)\sum_{0<|m|\leq (q\tau)^{1+\varepsilon}/M}J(m,q).
\eea
Note that the second term in (\ref{mnot0}) is bounded by
\bea
q^{-1}M^{1-\sigma_0}N^{\sigma_0+\theta+\varepsilon}(q\tau)^{1+\varepsilon}/M\ll \tau^{1+\varepsilon}N^{\sigma_0+\theta+\varepsilon}M^{-\sigma_0}.
\eea
Removing the condition $(n,q)=1$ in the first term in (\ref{mnot0}) at a cost of \ $O(\tau^{1+\varepsilon}N^{\sigma_0+\theta+\varepsilon}M^{-\sigma_0})$ and combining (7.8)-(7.11), we obtain
\bea\label{4}
&&T(\pm,M,N)=M^{1-\overline{s_0}}\sum_{n\geq1}\frac{\lambda_f(n)}{n^{1-s_0}}W_{1-s_0}\left(\frac{n}{q}\right)\mathcal{W}_2\left(\frac{n}{N}\right)
\sum_{0<|m|\leq (q\tau)^{1+\varepsilon}/M}J(m,q)e\left(\frac{\mp n\overline{m}}{q}\right)\notag\\
&+&O\left(\tau^{1+\varepsilon}N^{\sigma_0+\theta+\varepsilon}
M^{-\sigma_0}+q^{-1}M^{1-\sigma_0}N^{\sigma_0+\theta+\varepsilon}\right).
\eea
By Lemma \ref{e} and partial summation, we have
\bea\label{3}
&&\sum_{n\geq1}\frac{\lambda_f(n)}{n^{1-s_0}}W_{1-s_0}\left(\frac{n}{q}\right)\mathcal{W}_2\left(\frac{n}{N}\right)
e\left(\frac{\mp n\overline{m}}{q}\right)\notag\\
&=&\int_{N}^{2N}t^{s_0-1}W_{1-s_0}\left(\frac{t}{q}\right)\mathcal{W}_2\left(\frac{t}{N}\right)\mathrm{d}\left(\sum_{n\leq t}\lambda_f(n)e\left(\frac{\mp n\overline{m}}{q}\right)\right)\notag\\
&=&-\int_{N}^{2N}\sum_{n\leq t}\lambda_f(n)e\left(\frac{\mp n\overline{m}}{q}\right)\left(t^{s_0-1}W_{1-s_0}\left(\frac{t}{q}\right)\mathcal{W}_2\left(\frac{t}{N}\right)\right)^{\prime}\mathrm{d}t\notag\\
&\ll&\tau N^{\sigma_0-1/2+\varepsilon} \ll\tau(q\tau)^{\sigma_0-1/2+\varepsilon}.
\eea
Thus by (\ref{4}) and (\ref{3}),
\bea\label{case2}
T(\pm,M,N)&\ll&\ M^{1-\sigma_0}\frac{(q\tau)^{1+\varepsilon}}{M}\tau(q\tau)^{\sigma_0-\frac{1}{2}+\varepsilon}
+\tau^{1+\varepsilon}N^{\sigma_0+\theta+\varepsilon}M^{-\sigma_0}+
q^{-1}M^{1-\sigma_0}N^{\sigma_0+\theta+\varepsilon}\notag\\
&\ll&\tau(q\tau)^{(1-\beta_1)\sigma_0+\frac{1}{2}+\varepsilon}+\tau(q\tau)^{(1-\beta_1)\sigma_0+\theta+\varepsilon}
+q^{-1}(q\tau)^{(1-\beta_1)\sigma_0+\beta_1+\theta+\varepsilon}\notag\\
&\ll&\tau(q\tau)^{(1-\beta_1)\sigma_0+\frac{1}{2}+\varepsilon}+q^{-1}(q\tau)^{(1-\beta_1)\sigma_0+\beta_1+\theta+\varepsilon}.
\eea
Here we recall that $\theta=\frac{7}{64}$.\\
{\bf Case III:} $M <(q\tau)^{\beta_{1}} $, $N\geq(q\tau)^{\beta_{2}} $.\\
Note that
\bna
&&\sum_{ (m,q)=1}\frac{1}{m^{\overline{s_0}}}V_{\overline{s_0}}\left(\frac{m}{\sqrt{q}}\right)\mathcal{W}_1\left(\frac{m}{M}\right)\sum_{ q\mid n}\frac{\lambda_f(n)}{n^{1-s_0}}W_{1-s_0}\left(\frac{n}{q}\right)\mathcal{W}_2\left(\frac{n}{N}\right)S(1,\pm mn ;q)\notag \\
&\ll&M^{1-\sigma_0}q^{-1}N^{\sigma_0+\theta}\sqrt{q}\notag \\
&\ll& q^{-\frac{1}{2}}(q\tau)^{\beta_1(1-\sigma_0)}
(q\tau)^{\sigma_0+\theta+\varepsilon}\notag \\
&\ll& q^{-\frac{1}{2}}(q\tau)^{\beta_1(1-\sigma_0)+\sigma_0+\theta+\varepsilon},
\ena
where we have used Weil's bound for Kloosterman sums. Thus we can write (\ref{T}) as
\bea\label{TT*}
T(\pm,M,N)=T^{\ast}(\pm,M,N)+O\left(q^{-\frac{1}{2}}(q\tau)^{\beta_1(1-\sigma_0)+\sigma_0+\theta+\varepsilon}\right),
\eea
where
\bea\label{T*}
T^{\ast}(\pm,M,N)&=&
\sum_{(m,q)=1}\frac{1}{m^{\overline{s_0}}}V_{\overline{s_0}}\left(\frac{m}{\sqrt{q}}\right)\mathcal{W}_1\left(\frac{m}{M}\right)\notag \\
&\times&\sum_{n}\frac{\lambda_f(n)}{n^{1-s_0}}W_{1-s_0}\left(\frac{n}{q}\right)\mathcal{W}_2\left(\frac{n}{N}\right)S(1,\pm mn ;q).
\eea

Opening the Kloosterman sum and applying Voronoi summation formula in Lemma \ref{Vor} to the sum over $n$, we get
\bea\label{5}
&&\sum_{n}\frac{\lambda_f(n)}{n^{1-s_0}}W_{1-s_0}\left(\frac{n}{q}\right)\mathcal{W}_2\left(\frac{n}{N}\right)
S(1,\pm mn;q)\notag\\
&=&qN^{s_0-1}\sideset{}{^\ast}\sum_{a \bmod q}e\left(\frac{a}{q}\right)\sum_{\pm}\sum_{n\geq1}\frac{\lambda_f(n)}{n}e\left(\pm\frac{ a\overline{m}n}{q}\right)\Psi^{\pm}\left(\frac{nN}{q^{2}}\right),
\eea
where for $\sigma>-1$,
\bea\label{psi}
\Psi^{\pm}(y)=\frac{1}{2\pi i}\int_{(\sigma)}(\pi^{2}y)^{-s}G^{\pm}(s)
\left(\int_{0}^{\infty}u^{s_0-1}W_{1-s_0}\left(\frac{Nu}{q}\right)\mathcal{W}_2(u)u^{-s-1}\mathrm{d}u\right)\mathrm{d}s.
\eea
Note that the sum over $a$ equals
\bna
\sideset{}{^\ast}\sum_{a   \bmod q}
e\left(\frac{(1\pm n\overline{m})a}{q}\right)=q\mathbf{1}_{n\equiv \mp m (\bmod q)}-1.
\ena
Thus (\ref{5}) is equal to
\bea\label{qN}
qN^{s_0-1}\sum_{\pm}\sum_{n\geq 1}\frac{\lambda_f(n)}{n}\Psi^{\pm}\left(\frac{nN}{q^{2}}\right)
\left(q\mathbf{1}_{n\equiv \mp m (\bmod q)}-1\right).
\eea
Moreover, by repeated integration by parts, the $u$-integral in (\ref{psi}) is bounded by $\left(\frac{\tau}{1+|s|}\right)^j$ for any integer $j\geq0$. Therefore, by (\ref{G}) and Stirling's formula,
\bna
\Psi^{\pm}(y)\ll y^{-\sigma}\int_{|t|\leq \tau^{1+\varepsilon}q^\varepsilon}(1+|t|)^{2\sigma+1}\left(\frac{\tau}{1+|t|}\right)^j\mathrm{d}t+(q\tau)^{-A}
\ena
for any $A>0$. Taking $\sigma=j/2-1-\varepsilon$ with any fixed $j\geq1$, one has
\bna
\Psi^{\pm}(y)\ll y^{-j/2+1+\varepsilon}\tau^j\int_{|t|\leq \tau^{1+\varepsilon}q^\varepsilon}(1+|t|)^{-1-2\varepsilon}\mathrm{d}t+(q\tau)^{-A}\ll y^{1+\varepsilon}\left(\frac{y}{\tau^2}\right)^{-j/2}
\ena
for any fixed $j\geq1$. Therefore, the contribution from $\frac{nN}{q^2\tau^2}\geq(q\tau)^\varepsilon$ in (\ref{qN}) is negligible. For smaller $y$, we move the contour of integration in $\Psi^{\pm}(y)$ to $\Re(s)=-1+\varepsilon$ to get
\bea\label{psiy}
\Psi^{\pm}(y)\ll y^{1-\varepsilon}\int_{|t|\leq \tau^{1+\varepsilon}q^\varepsilon}(1+|t|)^{-1+\varepsilon}\mathrm{d}t\ll (q\tau)^{\varepsilon}y^{1-\varepsilon}.
\eea
By putting (\ref{qN}) into (\ref{T*}), we obtain
\bna
T^{\ast}(\pm,M,N)&=&qN^{s_0-1}\sum_{\pm}\sum_{(m,q)=1}\frac{1}{m^{\overline{s_0}}}V_{\overline{s_0}}\left(\frac{m}{q}\right)
\mathcal{W}_1\left(\frac{m}{M}\right)\\
&\times&\sum_{n\leq(q\tau)^{2+\varepsilon}/N}\frac{\lambda_f(n)}{n}\Psi^{\pm}\left(\frac{nN}{q^{2}}\right)
\left(q\mathbf{1}_{n\equiv \mp m (\bmod q)}-1\right).
\ena

Using (\ref{psiy}) and Lemma \ref{1}, we have
\bea\label{6}
T^{\ast}(\pm,M,N)&\ll& q^2N^{s_0-1}\sum_{m\asymp M}m^{-\sigma_0}\frac{|\lambda_f(m)|}{m}\cdot\frac{m N}{q^{2}}\notag\\
&+&q^2N^{\sigma_0-1}\sum_{m\asymp M}m^{-\sigma_0}\sum_{1\leq k\leq \tau(q\tau)^{1+\varepsilon}/N \atop n=\pm m+kq}n^{\theta-1}\cdot\frac{n N}{q^{2}}\notag\\
&+&qN^{\sigma_0-1}\sum_{m\asymp M}m^{-\sigma_0}\sum_{n\leq (q\tau)^{2+\varepsilon}/N }\frac{|\lambda_f(n)|}{n}\cdot\frac{n N}{q^{2}}\notag\\
&\ll&N^{\sigma_0}M^{1-\sigma_0+\varepsilon}+q\tau^2N^{\sigma_0-1}
M^{1-\sigma_0+\varepsilon}\frac{(q\tau)^{2\theta}}{N^\theta}+q\tau^2N^{\sigma_0-1}
M^{1-\sigma_0+\varepsilon}\notag\\
&\ll&N^{\sigma_0-1}M^{1-\sigma_0+\varepsilon}q\tau^2\frac{(q\tau)^{2\theta}}{N^\theta}\notag\\
&\ll&\tau(q\tau)^{1+(1-\sigma_0)\beta_1-(1-\sigma_0)\beta_2+(2-\beta_2)\theta+\varepsilon}.
\eea
By (\ref{TT*}) and (\ref{6}), we get
\bea\label{case3}
T(\pm,M,N)&\ll& \tau(q\tau)^{1+(1-\sigma_0)\beta_1-(1-\sigma_0)\beta_2+(2-\beta_2)\theta+\varepsilon}+
q^{-\frac{1}{2}}(q\tau)^{\beta_1(1-\sigma_0)+\sigma_0+\theta+\varepsilon}.
\eea
Note that $(2-\beta_2)\theta>\theta$ for $\beta_2<1$ and $1-(1-\sigma_0)\beta_2>\sigma_0$. The second term in (\ref{case3}) is dominated by the first term. Thus
\bea\label{case33}
T(\pm,M,N)&\ll& \tau(q\tau)^{1+(1-\sigma_0)\beta_1-(1-\sigma_0)\beta_2+(2-\beta_2)\theta+\varepsilon}.
\eea
Therefore, Lemma \ref{last lemma} follows from (\ref{case1}), (\ref{case2}) and (\ref{case33}).\
\end{proof}
By Lemma \ref{last lemma} and (\ref{S_41}),
\bea\label{S_41r}
S_{41}&\ll& \tau^{-\frac{1}{2}}(q\tau)^{\frac{1}{2}+(1-\sigma_0)\beta_{1}+\beta_{2}\sigma_0+\varepsilon}
+\tau(q\tau)^{(1-\beta_1)\sigma_0+\frac{1}{2}+\varepsilon}\notag \\
&+&\tau(q\tau)^{1+(1-\sigma_0)\beta_1-(1-\sigma_0)\beta_2+(2-\beta_2)\theta+\varepsilon}.
\eea
Solving the equations
\bna
\left\{
\begin{aligned}
&&\frac{1}{2}+(1-\sigma_0)\beta_{1}+\beta_{2}\sigma_0=(1-\beta_1)\sigma_0+\frac{1}{2}, \\
&&1+(1-\sigma_0)\beta_1-(1-\sigma_0)\beta_2+(2-\beta_2)\theta=(1-\beta_1)\sigma_0+\frac{1}{2},
\end{aligned}
\right.
\ena
 we choose $\beta_1$ and $\beta_2$ as
\bea\label{beta}
\beta_1=\frac{\sigma_0(1-2\theta)}{2(1+\theta)}, \quad \quad \beta_2=\frac{1+4\theta}{2(1+\theta)}.
\eea
Thus by (\ref{S_41r}),
\bea
S_{41}\ll \label{S_41rr}
\tau(q\tau)^{\frac{1}{2}+\left(1-\frac{\sigma_0(1-2\theta)}{2(1+\theta)}\right)\sigma_0+\varepsilon}.
\eea
By Stirling's formula,
\bea\label{jmh1}
\frac{\widetilde{\gamma}(1-s_0)}{\widetilde{\gamma}(s_0)}\ll \tau^{1-2\sigma_0}.
\eea
Substituting (\ref{S_42}), (\ref{beta})-(\ref{jmh1}) into (\ref{S_4}), we conclude that
\bna
S_4&\ll&(q\tau)^{1-2\sigma_0}
\tau(q\tau)^{\frac{1}{2}+\left(1-\frac{\sigma_0(1-2\theta)}{2(1+\theta)}\right)\sigma_0+\varepsilon}
+q^{-2\sigma_0}\tau^{1-2\sigma_0}(q\tau)^{\frac{1}{2}+\frac{\sigma_0}{2}+\theta+\varepsilon}\notag\\
&\ll&
\tau(q\tau)^{\frac{3}{2}-\left(1+\frac{\sigma_0(1-2\theta)}{2(1+\theta)}\right)\sigma_0+\varepsilon}
+\tau(q\tau)^{\frac{1}{2}-\frac{3}{2}\sigma_0+\theta+\varepsilon}.
\ena
Note that $1+\frac{\sigma_0(1-2\theta)}{2(1+\theta)}<\frac{3}{2}$ and $\frac{3}{2}>\frac{1}{2}+\theta$ for $\theta\leq \frac{7}{64}$. Therefore, the first term dominates the second term and
\bea\label{xy}
S_4\ll \tau(q\tau)^{\frac{3}{2}-\left(1+\frac{\sigma_0(1-2\theta)}{2(1+\theta)}\right)\sigma_0+\varepsilon}.
\eea
\section{Completing the proof of Theorem 1}
By (\ref{sum}), Lemma \ref{jmh}-\ref{jmhhh} and (\ref{xy}), we have
\bea
\sideset{}{^\dagger}\sum_{\chi \text(\bmod q) \atop \chi(-1)=1}
L(s_0,f\otimes\chi)\overline{L(s_0,\chi)}=\frac{q}{2}L(2\sigma_0,f)+E,
\eea
where
\bna
E&\ll&q^{\frac{1}{4}+\varepsilon}\tau^{\frac{3}{2}(1-\sigma_0)}+(q\tau)^{\frac{3}{2}(1-\sigma_0)+\theta+\varepsilon}
+q^{\frac{1}{2}}(q\tau)^{-\frac{1}{2}\sigma_0+\frac{1}{2}+\varepsilon}\\
&+&q^{\frac{1}{2}}(q\tau)^{1-\frac{3}{2}\sigma_0+\varepsilon}
+\tau(q\tau)^{\frac{3}{2}-\left(1+\frac{\sigma_0(1-2\theta)}{2(1+\theta)}\right)\sigma_0+\varepsilon}.
\ena
Note that for $\frac{1}{2}\leq\sigma_0<1$, the third term domiantes the fourth term. Therefore,
\bea
E&\ll&q^{\frac{1}{4}+\varepsilon}\tau^{\frac{3}{2}(1-\sigma_0)}+(q\tau)^{\frac{3}{2}(1-\sigma_0)+\theta+\varepsilon}
+q^{\frac{1}{2}}(q\tau)^{\frac{1}{2}(1-\sigma_0)+\varepsilon}\notag\\
&+&\tau(q\tau)^{\frac{3}{2}-\left(1+\frac{\sigma_0(1-2\theta)}{2(1+\theta)}\right)\sigma_0+\varepsilon}.
\eea
By (8.1) and (8.2), Theorem 1 follows.

\end{document}